\documentclass[12pt]{amsart}

\usepackage{amsmath,amsthm,amsfonts,amssymb,eucal, bbm}
\usepackage{scalerel}

\usepackage{color}

\renewcommand {\a}{ \alpha }
\renewcommand{\b}{\beta}

\newcommand{\g}{\gamma}

\renewcommand{\l}{\lambda}
\renewcommand{\L}{\Lambda}
\newcommand{\z}{\zeta}
\renewcommand{\t}{\theta}

\newcommand{\p}{\partial}
\newcommand{\om}{\omega}
\newcommand{\Om}{\Omega}

\newcommand{\oq}{\ {\raise 7pt\hbox{${\scriptstyle\circ}$}}
	\kern -7pt{
		\hbox{$Q$}}}

\newcommand{\R}{ \mathbb R}

\newcommand {\ba}{\mathbf a}

\newcommand {\BA}{\mathbf A}

\newcommand {\BS}{\mathbf S}

\newcommand {\bx}{\mathbf x}

\newcommand {\by}{\mathbf y}

\newcommand{\SV}{{\sf V}}

\newcommand{\lu}{\langle}
\newcommand{\ru}{\rangle}

\newcommand{\CH}{\mathcal H}

\newcommand{\CC}{\mathcal C}

\newcommand{\plainW}[1]{\textup{{\textsf{W}}}^{#1}}
\newcommand{\plainC}[1]{\textup{{\textsf{C}}}^{#1}}

\newcommand{\plainL}[1]{\textup{{\textsf{L}}}^{#1}}

\newcommand{\scalel}[1]
{{\scaleto{#1}{3pt}}}
\newcommand{\scalet}[1]
{{\scaleto{#1}{4pt}}}

\newcommand{\4}{\ \!{\vrule depth3pt height9pt}
	{\vrule depth3pt height9pt}
	{\vrule depth3pt height9pt}{\vrule depth3pt height9pt}}

\DeclareMathOperator {\dist} {{dist}}

\DeclareMathOperator{\iop}{{\sf Int}}

\DeclareMathOperator{\dc}{d}

\newtheorem{thm}{Theorem}[section]
\newtheorem{cor}[thm]{Corollary}
\newtheorem{lem}[thm]{Lemma}
\newtheorem{prop}[thm]{Proposition}

\theoremstyle{definition}

\newtheorem{rem}[thm]{Remark}

\numberwithin{equation}{section}

\newcommand{\bee}{\begin{equation}}
	\newcommand{\ene}{\end{equation}}
\newcommand{\bees}{\begin{equation*}}
	\newcommand{\enes}{\end{equation*}}
\newcommand{\bes}{\begin{split}}
	\newcommand{\ens}{\end{split}}

\newcommand{\bet}{\begin{thm}}
	\newcommand{\ent}{\end{thm}}
\newcommand{\bel}{\begin{lem}}
	\newcommand{\enl}{\end{lem}}
\newcommand{\bec}{\begin{cor}}
	\newcommand{\enc}{\end{cor}}
\newcommand{\bep}{\begin{proof}}
	\newcommand{\enp}{\end{proof}}
\newcommand{\ber}{\begin{rem}}
	\newcommand{\enr}{\end{rem}}

\newcommand{\Z}{\mathbb Z}

\newcommand{\1}{\mathbbm 1}

\begin{document}
	\hoffset -4pc

\title
[ Density matrix]
{{Eigenvalue estimates for the coulombic one-particle density matrix 
and the kinetic energy density matrix}} 
\author{Alexander V. Sobolev}
\address{Department of Mathematics\\ University College London\\
	Gower Street\\ London\\ WC1E 6BT UK}
 \email{a.sobolev@ucl.ac.uk}

\dedicatory{To the memory of Dima Yafaev}

\begin{abstract} 
Consider a bound state (an eigenfunction) $\psi$ of an atom with $N$ electrons.   
We study the spectra of the  one-particle 
density matrix $\g$ and of the one-particle 
kinetic energy density matrix $\tau$ associated with $\psi$. 
The paper contains two results. First, we obtain the bounds 
$\l_k(\g)\le C_1 k^{-8/3}$ and $\l_k(\tau)\le C_2 k^{-2}$ 
with some positive constants $C_1, C_2$ that depend explicitly on the  
eigenfunction $\psi$. The sharpness of these bounds is confirmed by the 
asymptotic results obtained by the author in earlier papers. 
The advantage of these bounds over the ones derived by the author previously, 
is their explicit dependence on the eigenfunction. Moreover, their new 
proofs are more elementary and direct. 
The second result is new and it pertains to the case 
where the eigenfunction $\psi$ vanishes at the particle 
coalescence points. In particular, this is true 
for totally antisymmetric $\psi$. In this case 
the eigenfunction $\psi$ exhibits enhanced 
regularity at the coalescence points which leads to the 
faster decay of the eigenvalues:  
$\l_k(\g)\le C_3 k^{-10/3}$ and $\l_k(\tau)\le C_4 k^{-8/3}$. 

The proofs rely on the estimates for the derivatives of the eigenfunction $\psi$ 
that depend explicitly on the distance to the coalescence points. 
Some of these estimates are 
borrowed directly from, and some are derived using the methods of a recent paper by 
S. Fournais and T. \O. S\o rensen.
\end{abstract}

\maketitle

\section{Introduction}

Consider on $\plainL2(\R^{3N})$ the Schr\"odinger operator 
\begin{align}
\CH = &\ \CH_0 + V,\quad \CH_0 = - \Delta = - \sum_{k=1}^N \Delta_k,\notag\\
V(\bx) = &\ - Z
\sum_{k=1}^N \frac{1}{|x_k|}  
 + \sum_{1\le j< k\le N} \frac{1}{|x_j-x_k|},\label{eq:potential}
\end{align}
describing an atom with $N$ particles 
(e.g. electrons)  
with coordinates $\bx = (x_1, x_2, \dots, x_N)$, $x_k\in\R^3$, $k= 1, 2, \dots, N$, 
and a nucleus with charge $Z>0$.  
The notation $\Delta_k$ is used for 
the Laplacian w.r.t. the variable $x_k$. 
The operator $\CH$ acts on the Hilbert space $\plainL2(\R^{3N})$ and by
 standard methods one proves that it is self-adjoint on the domain 
$D(\CH) =\plainW{2, 2}(\R^{3N})$, 
see e.g. \cite[Theorem X.16]{ReedSimon2}. (Here and throughout the paper we use the standard notation 
$\plainW{l, p}$ for the Sobolev spaces, where $l$ and $p$ indicate the smoothness and summability 
respectively). 
Our methods allow consideration of the molecular 
Schr\"odinger operator, but we restrict our attention to the atomic case for simplicity. 
Let $\psi = \psi(\bx)$,  
be an eigenfunction of the operator $\CH$ with an eigenvalue $E\in\R$, i.e. $\psi\in D(\CH)$ and 
\begin{align}\label{eq:eigen}
(\CH-E)\psi = 0.
\end{align}
For each $j=1, \dots, N$, we represent
\begin{align*}
\bx = (\hat\bx_j, x_j), \quad \textup{where}\ 
\hat\bx_j = (x_1, \dots, x_{j-1}, x_{j+1},\dots, x_N),
\end{align*}
with obvious modifications if $j=1$ or $j=N$. 
The \textit{one-particle density matrix} is defined as the function 
\begin{align}\label{eq:den}
\g_0(x, y) = \sum_{j=1}^N\int\limits_{\R^{3N-3}}\overline{\psi(\hat\bx_j, x)} \,
 \psi(\hat\bx_j, y)\,  d\hat\bx_j,\  \textup{a.e.}\ x\in\R^3,
 \ \textup{a.e.}\  y\in\R^3. 
\end{align} 
Note that traditionally, the one-particle density matrix is defined as  
$\g_0(y, x)$, see e.g. \cite{LLS2019}, but the permutation $x\leftrightarrow y$ 
will have no bearing on our results. 

Introduce also the function 
\begin{align}\label{eq:tau}
\tau_0(x, y) = 
\sum_{j=1}^N\int\limits_{\R^{3N-3}} \overline{\nabla_x\psi(\hat\bx_j, x)} 
\cdot\nabla_y \psi(\hat\bx_j, y)\  
d\hat\bx_j, 
\end{align}
that we call the \textit{one-particle kinetic energy density matrix}. 
The choice of this term does not seem to be standard, but it is partly 
motivated by the fact that the integral
$\int_{\R^3} \tau_0(x, x)\, dx $
gives the kinetic energy of the $N$ particles, see e.g. 
\cite[Section 2A]{Davidson1976}, \cite[Chapter 3]{LiebSei2010} or 
\cite[Section 4]{LLS2019}. 
If we assume that the particles are spinless fermions (resp. bosons), i.e. that the eigenfunction 
is totally antisymmetric 
(resp. symmetric) with respect to the permutations $x_k\leftrightarrow x_j$, 
$j, k = 1, 2, \dots, N, j\not = k$, 
then the formulas for $\g_0(x, y)$ and $\tau_0(x, y)$ take a more compact form:
\begin{align*}
\g_0(x, y) = &\ N \int\limits_{\R^{3N-3}} 
\overline{\psi(\hat\bx, x)}\, \psi(\hat\bx, y)\ d\hat\bx,\\
\tau_0(x, y) = &\ N\int\limits_{\R^{3N-3}} \overline{\nabla_x\psi(\hat\bx, x)} 
\cdot\nabla_y \psi(\hat\bx, y)\,d\hat\bx,
\end{align*}
where $\hat\bx = \hat\bx_N$. 

We consider $\g_0$ and $\tau_0$ as integral kernels of self-adjoint non-negative operators that we denote 
$\iop(\g_0)$ and $\iop(\tau_0)$ respectively. 
Since $\psi, \nabla\psi\in \plainL2(\R^{3N})$, these operators 
are trace class. We are interested in the precise decay rate of 
their eigenvalues $\l_k(\iop(\g_0))$, $\l_k(\iop(\tau_0))>0$, 
$k = 1, 2, \dots$, enumerated in non-increasing order. 
%
Recall that in the physics literature 
the eigenvalues $\l_k(\iop(\g_0))$ are 
called  \textit{occupation numbers}, see e.g. \cite{SzaboOstlund1996}.  
Both functions $\g_0$ and $\tau_0$ are 
key objects in multi-particle quantum mechanics, see 
\cite{RDM2000, Davidson1976, LLS2019, LiebSei2010, SzaboOstlund1996} 
for details and futher references, 
and both operators $\iop(\g_0)$ and $\iop(\tau_0)$ play a central 
role in quantum chemistry computations of atomic and molecular bound states, see e.g.    
\cite{Fries2003_1, Fries2003, Lewin2004, Lewin2011, SzaboOstlund1996}. 
The knowledge of the eigenvalue behaviour should 
serve to estimate the errors due to finite-dimensional 
approximations,  see e.g. 
\cite{Cioslowski2020, CioStras2021,  Fries2003, HaKlKoTe2012}.

The eigenvalue asymptotics for the operator $\iop(\g_0)$ in the case $N=2$ were discussed 
in \cite{Cioslowski2020, CioPrat2019}, and it was found that $\l_k(\iop(\g_0))$  
behave like $k^{-8/3}$ for large $k$. The case of arbitrary $N$ 
was studied in \cite{Sobolev2022} (for $\iop(\g_0)$) 
and \cite{Sobolev2022a} (for $\iop(\tau_0)$) 
under the assumption that the function $\psi$ satisfies an exponential bound 
\begin{align}\label{eq:exp}
\sup_{\bx\in\R^{3N}} e^{c |\bx|}\, 
 |\psi(\bx)|<\infty
\end{align}  
with some constant $c>0$.  
For discrete eigenvalues $E$ under 
the essential spectrum of $\CH$, the bound \eqref{eq:exp} follows from 
\cite{DHSV1978_79}. It also holds for embedded eigenvalues as long as they are away from the so-called \textit{thresholds}, see \cite{CombesThomas1973}, \cite{FH1982}. 
For more references and detailed discussion we quote \cite{SimonSelecta}. 
It was shown by the author in \cite{Sobolev2022, Sobolev2022a} that 
\begin{align}\label{eq:asymp}
\lim_{k\to \infty} k^{\frac{8}{3}} \,\l_k(\iop(\g_0)) = A^{\frac{8}{3}},\quad 
\lim_{k\to \infty} k^2 \,\l_k(\iop(\tau_0)) = B^2
\end{align}
with some coefficients $A\ge 0, B\ge 0$.
One central ingredient in the proof of \eqref{eq:asymp} was the sharp eigenvalue bounds 
obtained in \cite{Sobolev2022b} (for $\iop(\g_0)$) and \cite{Sobolev2022a} 
(for $\iop(\tau_0)$), under the same condition \eqref{eq:exp}. 
The purpose of the current paper is to offer a short elementary proof of 
the eigenvalue bounds and to derive ``improved" bounds for the case where the eigenfunction 
vanishes at the particle coalescence points. 
It is clear that it suffices to 
 study each term in \eqref{eq:den} and \eqref{eq:tau} individually. 
Moreover, using permutations of the variables it is sufficient to focus just on the last 
term in the sum \eqref{eq:den} or \eqref{eq:tau}:
\begin{align*}
\g(x, y) = &\ \int_{\R^{3N-3}}
\overline{\psi(\hat\bx, x)} \, \psi(\hat\bx, y) \, d\hat\bx,\\
\tau(x, y) = &\ \int_{\R^{3N-3}}
\overline{\nabla_x\psi(\hat\bx, x)}\cdot \nabla_y\psi(\hat\bx, y) \, d\hat\bx,\ 
\quad \hat\bx = (x_1, x_2, \dots, x_{N-1}).
\end{align*}
Throughout the paper we refer to these functions as the  one-particle density matrix 
and the  one-particle kinetic energy density matrix, instead of \eqref{eq:den} 
and \eqref{eq:tau}. 
The main results are stated in terms of the \textit{one-particle density} 
\begin{align}\label{eq:dens}
\rho(x) = \g(x, x) = \int_{\R^{3N-3}}
|\psi(\hat\bx, x)|^2 \, d\hat\bx.
\end{align}
Let $\CC = [-1/2, 1/2)^3\in \R^3$ be a unit cube, and let $\CC_n = \CC+n, n\in \mathbb Z^3,$ 
be its translations by integer vectors.  For $q >0$ and a function 
$f\in \plainL{1}_{\rm loc}(\R^3)$ define its \textit{lattice (semi-)norm}
\begin{align*}
\4\, f\4_{q} = \bigg[\sum_{n\in \Z^3} \, \|f\|_{\plainL1(\CC_n)}^q\bigg]^{\frac{1}{q}}.
\end{align*}
This functional defines a norm (if $q\ge 1$) or a semi-norm (if $q <1$). 
The next two theorems constitute the main results of the paper.

\begin{thm}\label{thm:main}
Let $\psi\in\plainL2(\R^{3N})$ 
be an eigenfunction of the operator $\CH$ defined in \eqref{eq:potential} 
with the eigenvalue $E$. Then 
\begin{align}
\l_k\big(\iop(\g)\big)\le &\ C k^{-\frac{8}{3}} \, \4\, \rho\4\,_{3/8},\label{eq:gbound}\\
\l_k\big(\iop(\tau)\big)\le &\ C k^{-2} \, \4\, \rho\4\,_{1/2}.\label{eq:taubound}
\end{align}
The constants $C>0$ in the above bounds do not 
depend on the eigenfunction $\psi$ but may depend on 
$E$ and $N$.
\end{thm}

Under the assumption \eqref{eq:exp} the bounds 
$\l_k\big(\iop(\g)\big)\le  C k^{-\frac{8}{3}} $ 
and $\l_k\big(\iop(\tau)\big)\le C k^{-2}$ 
with constants depending on $\psi$, $N$ and $E$, 
were proved in \cite{Sobolev2022b} and \cite{Sobolev2022a} 
respectively. One advantage of \eqref{eq:gbound} and \eqref{eq:taubound} is 
their explicit dependence on the eigenfunction $\psi$. 
The right-hand sides of \eqref{eq:gbound} and \eqref{eq:taubound} are certainly 
finite under the assumption \eqref{eq:exp}, but an appropriate 
polynomial decay of $\psi$ would be sufficient. At the same time, the condition $\psi\in\plainL2(\R^{3N})$ 
is not enough. 

The next theorem provides eigenvalue bounds under the additional assumption that 
the eigenfunction $\psi$ vanishes at the particle coalescence points:
\begin{align}\label{eq:van}
\psi(\bx) = 0\quad \textup{if}\quad  x  = x_k, \quad \textup{for every}\  
k = 1, 2, \dots, N-1.
\end{align} 
This condition is satisfied, e.g. if the function  $\psi$ is totally antisymmetric, 
i.e. all particle pairs are in the triplet configuration, see \cite[Sect. 3.3.2]{HaKlKoTe2012}

\begin{thm}\label{thm:main0} 
Suppose that the eigenfunction $\psi$ satisfies \eqref{eq:van}. Then 
\begin{align}
\l_k\big(\iop(\g)\big)\le &\ C k^{-\frac{10}{3}} \, \4\, \rho\4\,_{3/10},\label{eq:gbound0}\\
\l_k\big(\iop(\tau)\big)\le &\ C k^{-\frac{8}{3}} \, \4\, \rho\4\,_{3/8}.\label{eq:taubound0}
\end{align}
The constants $C>0$ in the above bounds do not depend on the eigenfunction 
$\psi$ but may depend on $E$ and $N$.
\end{thm}

To the best of the author's knowledge, 
these bounds are new. Observe that the assumption \eqref{eq:van} leads 
to a faster decay of the eigenvalues than in Theorem \ref{thm:main}. 
As in Theorem \ref{thm:main}, under the condition 
\eqref{eq:exp} the right-hand sides of the bounds are finite.

\begin{rem}\label{rem:main}
\begin{enumerate}
 \item 
Theorems \ref{thm:main}, \ref{thm:main0} extend 
to the case of a molecule with several nuclei whose positions are fixed. 
The modifications are straightforward. 
\item 
One can also study the eigenvalues of the 
 $n$-particle density matrix with arbitrary $n = 1, 2, \dots, N-1$, 
 see \cite[Section 3.1.5]{LiebSei2010} for the definition. 
Some results in this direction were obtained in \cite{Hearnshaw2024} using the approach of 
\cite{Sobolev2022a}. 
 \end{enumerate}
\end{rem}

As in \cite{Sobolev2022a} and \cite{Sobolev2022b}, the first step in the proof of the two main theorems is to factorize the operators 
$\iop(\g)$ and $\iop(\tau)$ as follows: 
\begin{align}\label{eq:factor}
\iop(\g) = \Psi^*\Psi,\ \quad \iop(\tau) = \SV^*\SV,
\end{align}
where $\Psi: \plainL2(\R^3)\mapsto \plainL2(\R^{3N-3})$ and 
$\SV:\plainL2(\R^3)\mapsto \plainL2(\R^{3N-3}; \mathbb C^3)$ are the operators
\begin{align}\label{eq:psiv}
(\Psi f)(\hat\bx) = \int\psi(\hat\bx, x) f(x) \, dx,
\quad (\SV f)(\hat\bx) = \int\nabla_x \psi(\hat\bx, x) f(x) \, dx. 
\end{align}
Therefore $\l_k(\iop(\g)) = s_k(\Psi)^2$ and 
$\l_k(\iop(\tau)) = s_k(\SV)^2$, $k = 1, 2, \dots $, where 
$s_k(A)$ denote the singular values ($s$-values) of a compact operator $A$. Thus Theorems 
\ref{thm:main} and \ref{thm:main0} reduce to  estimating singular values of $\Psi$ and $\SV$. 

A general theory of spectral estimates for integral operators was developed in \cite{BS1977}. 
It demonstrates that the fall-off rate of the singular values 
increases together 
with the smoothness of the integral kernel. Quantitatively,  in \cite{BS1977}
the eigenvalues are estimated via suitable Sobolev or Besov norms of the kernels. 
The author applied such general estimates to  
the operators $\Psi$ and $\SV$ in  
\cite{Sobolev2022a} and \cite{Sobolev2022b}. In the current paper we suggest a direct, more elementary approach, 
which makes the proof of Theorems \ref{thm:main} and \ref{thm:main0} self-contained. 
We bypass the general bounds of \cite{BS1977} by 
directly estimating the singular values via the Fourier transform of $\psi(\hat\bx, x)$ in the variable $x$. 
In order to obtain suitable bounds for the Fourier transform one needs to control the regularity 
of $\psi$.   
Regularity of solutions for elliptic equations is a well-studied classical subject. 
In particular, due to the 
analyticity of the Coulomb potential $|x|^{-1}$ for $x\not = 0$, 
the classical results ensure that the function $\psi$ is real analytic away from the particle 
coalescence points. 
At the same time, as shown by 
T. Kato \cite{Kato1957}, the function $\psi$ is Lipschitz on the whole of $\R^{3N}$. 
Further regularity properties 
of $\psi$ were 
obtained, e.g.  in \cite{FHOS2005, FHOS2009, FS2021, HOS2001, HOSt1994}. 
The results of the paper 
\cite{FS2021} by S. Fournais and T. \O. S\o rensen, are of decisive importance for us, as they 
contain bounds on derivatives of $\psi$ of arbitrary order with 
explicit dependence on the distance of $x$ to the coalescence points. These results 
ensure a suitable Fourier transform estimates that lead to \eqref{eq:gbound} and \eqref{eq:taubound}. 

The estimates from \cite{FS2021} are not sufficient for the proof of 
Theorem \ref{thm:main0} however, since under the condition \eqref{eq:van} the 
eigenfunction $\psi(\hat\bx, x)$ has an improved regularity in the neighbourhood of 
the coalescence points. Precisely, 
in addition to being Lipschitz, it also has bounded second derivatives with respect to  
the variable $x$, see \cite{Hearnshaw2024}. 
Although this fact was not pointed out in \cite{FS2021}, it 
can be easily derived 
using the regularity result in \cite{FHOS2005}.  
We use this observation to obtain an ``improved" variant of the bounds from 
\cite{FS2021} using a slight modification 
of the method of \cite{FS2021} as proposed in \cite{HearnSob2023}. 
As a result, the Fourier transform of $\psi$ decays at infinity faster 
than that in Theorem \ref{thm:main} which entails 
a faster eigenvalue decay in Theorem \ref{thm:main0}.


%

The plan of the paper is as follows.
In Sect. \ref{sect:reg} we collect all the information on regularity of the function $\psi$ 
necessary for the proofs of the main theorems. 
A short Sect. \ref{sect:compact} gathers definitions and elementary properties of 
classes $\BS_{p, \infty}$ of compact operators. 
In Sect. \ref{sect:int} we derive spectral estimates for integral operators. These 
estimates are used to prove Theorems 
\ref{thm:main} and \ref{thm:main0} in Sect. \ref{sect:proof}.

We conclude the introduction with some general notational conventions.  

\textit{Coordinates.} 
As mentioned earlier, we use the following standard notation for the coordinates: 
$\bx = (x_1, x_2, \dots, x_N)$,\ where $x_j\in \R^3$, $j = 1, 2, \dots, N$. 
The vector $\bx$ is usually represented in the form 
$\bx = (\hat\bx, x)$ with  
$\hat\bx = (x_1, x_2, \dots, x_{N-1})\in\R^{3N-3}$,\, $x = x_N$.   
The notation $B(t, R)\subset\R^d$ is used for the open ball of radius $R>0$ centred at $t\in \R^d$, 
and $|B(t, R)|$ denotes its volume ($d$-dimensional Lebesgue measure).  
For instance, $B(\bx, R)\subset \R^{3N}$ and $B(\hat\bx, R)\subset \R^{3N-3}$.

\textit{Indicators.} For any set $\L\subset\R^d$ we denote by $\1_{\L}$ its indicator function (or indicator).
 
\textit{Derivatives.} 
Let $\mathbb N_0 = \mathbb N\cup\{0\}$.
If $x = (x', x'', x''')\in \R^3$ and $m = (m', m'', m''')\in \mathbb N_0^3$, then 
the derivative $\p_x^m$ is defined in the standard way:
\begin{align*}
\p_x^m = \p_{x'}^{m'}\p_{x''}^{m''}\p_{x'''}^{m'''}.
\end{align*} 
  
\textit{Bounds.} 
For two non-negative numbers (or functions) 
$X$ and $Y$ depending on some parameters, 
we write $X\lesssim Y$ (or $Y\gtrsim X$) if $X\le C Y$ with 
some positive constant $C$ independent of those parameters. 
If $X\lesssim Y\lesssim X$, then  we say that $X\asymp Y$. 
To avoid confusion we may sometimes comment on the nature of 
(implicit) constants in the bounds. 

\textit{Integral operators.}
We introduce the notation $\iop(T):\plainL2(\R^3)\mapsto \plainL2(\R^{3N-3})$   
for the integral operator with the kernel $T(\hat\bx, x)$. 

\section{Regularity of $\psi$}\label{sect:reg}
 
In this section we collect the pointwise 
bounds for derivatives of the eigenfunction $\psi$ of arbitrary order with explicit dependence 
on the distance to the coalescence points.  

\subsection{Bounds
}
Denote by $\Sigma$ the coalescence set:
\begin{align*}
\Sigma = \bigg\{\bx = (\hat\bx, x)\in\R^{3N}: |x| 
\prod_{1\le k\le N-1}|x_k-x| = 0\bigg\}.
\end{align*}
Then 
\begin{align}\label{eq:dq}
\dc(\bx) = \min\{|x|, \frac{1}{\sqrt 2}
|x-x_k|: \ 1\le k\le N-1\},
\end{align}
defines the distance from $\bx$ to $\Sigma$. 
Introduce also the capped distance $\l(\bx) := \min\{1, \dc(\bx)\}$.  
The functions $\dc(\,\cdot\,)$ and $\l(\,\cdot\,)$ are 
Lipschitz with Lipschitz constant $=1$ 
(see \cite[Lemma 4.2]{HearnSob2023}):  for all $\bx, \by\in\R^{3N}$ we have 
\begin{align}\label{eq:lip}
|\dc(\bx) - \dc(\by)|\le |\bx-\by|,\ \quad |\l(\bx) - \l(\by)|\le |\bx-\by|.
\end{align}
The following bound for the derivatives of $\psi$ 
is a consequence of \cite[Corollary 1.2]{FS2021}.  

\begin{prop}\label{prop:FS} 
For all $\bx\notin\Sigma$, and for 
arbitrary $R >0$, we have the bound 
\begin{align}\label{eq:FS}
|\p_{x}^{m}\psi(\hat\bx, x)|\lesssim \big(1+\l(\bx)^{1-|m|}\big)\, \|\psi\|_{\plainL2(B(\bx, R))}, 
\, \quad \textup{for all}\quad m\in\mathbb N_0^{3},
\end{align}
with a constant that does not depend on $\bx$ and $\psi$, but may depend on $m$ and $R$.
\end{prop}
 
The estimate \eqref{eq:FS} will be sufficient for the proof of Theorem \ref{thm:main}. 
With regard to Theorem \ref{thm:main0}, 
%
%
we shall see 
that under the assumption \eqref{eq:van}  the function $\psi$ has ``enhanced" regularity at the 
coalescence points $x = x_k$, $k = 1, 2, \dots, N-1$, 
but not at $x = 0$. To account for the singularity at $x = 0$ 
we introduce the Lipschitz function 
\begin{align}\label{eq:F0}
F_0(x) = -\frac{Z}{2}|x|\,\t(x),
\end{align}
where $\t\in \plainC\infty_0(\R^3)$ is such that $\t(x) = 1$ for $|x|<1$. 
Then we have the following result:

\begin{thm}\label{thm:varphider} Assume that \eqref{eq:van} holds. 
Then for any $R>0$ and for all $\bx\notin\Sigma$, we have   
\begin{align}\label{eq:varphider}
|\p_x^m\, \big(e^{-F_0(x)}\,\psi(\bx)\big)|
\lesssim (1+\l(\bx)^{2- |m|}) \, \|\psi\|_{\plainL2(B(\bx, R))}, 
\end{align}
for all $m\in\mathbb N_0^3$.
\end{thm}

The constants in \eqref{eq:varphider} depend on  
the function $\t$ but we assume that $\t$ is fixed throughout the paper 
and ignore the dependence on $\t$ in this and all the subsequent inequalities.

The rest of this section is focused on the proof of Theorem \ref{thm:varphider}. 
It requires detailed information 
on the structure of the function $\psi$, which is discussed next. 

\subsection{The Jastrow factors} 
The starting point of analysis in \cite{FHOS2005, FHOS2009, FS2021, HOS2001} 
and in many other papers on the subject, is the representation $\psi = e^F \phi$, where $F$ is a 
function that satisfies 
$F, \nabla F\in \plainL\infty_{\tiny {\rm loc}}(\R^{3N})$ and 
that captures the main singularity of the solution, so that 
$\phi$ is more regular than $\psi$. 
In the mathematical physics literature the function $e^F$ is often called 
a \textit{Jastrow factor}, see e.g. \cite{HaKlKoTe2012}.  
After the substitution the equation \eqref{eq:eigen} rewrites as 
\begin{align}\label{eq:jastrow}
-\Delta \phi - 2\nabla F\cdot \nabla\phi + (V- \Delta F - |\nabla F|^2-E)\phi = 0. 
\end{align}
More precisely,  if $\psi\in \plainW{1, 2}_{\tiny {\rm loc}}(\R^{3N})$ is a weak 
solution of \eqref{eq:eigen}, 
then we also have $\phi\in \plainW{1, 2}_{\tiny {\rm loc}}(\R^{3N})$ 
and $\phi$ is a weak solution of \eqref{eq:jastrow}.  
 If one chooses $F$ to be 
\begin{align}\label{eq:oldF}
\sum_{j=1}^N \bigg(-\frac{Z}{2}|x_j| 
+ \frac{1}{4}\sum_{j<k\le N}|x_j-x_k|\bigg),
\end{align} 
as in \cite{Leray1984}, 
then $\nabla F$ is bounded,  
$\Delta F = V$, and hence all the coefficients in \eqref{eq:jastrow} are uniformly bounded. 
The  function $F$ itself however is not bounded. To remedy this, 
we take the following regularized variant of \eqref{eq:oldF}. 
Let $\t\in \plainC\infty_0(\R^3)$ be as in the definition 
\eqref{eq:F0}. 
From now on we assume that $F$ is given by 
\begin{align}\label{eq:jascut}
F(\bx) = -\frac{Z}{2}\sum_{j=1}^N |x_j|\t(x_j) 
+ \frac{1}{4}\sum_{1\le j<k\le N}|x_j-x_k| \t(x_j-x_k), 
\end{align}
so that 
\begin{align}\label{eq:Finf}
F,\, \nabla F,\, \nabla^k (V-\Delta F)\in \plainL\infty(\R^{3N}),\quad \textup{for all}
\quad k = 0, 1, \dots.
\end{align} 
A different regularization (still satisfying \eqref{eq:Finf}) of the Jastrow factor was used in 
\cite{HOS2001}, \cite{FHOS2002}, \cite{FHOS2004} and \cite{FS2021}. 
The function $F(\hat\bx, x)$ naturally splits into two components:
$F(\hat\bx, x) = F_0(x) + F_1(\bx)$, where $F_0$ is as defined in \eqref{eq:F0} and 
\begin{align}\label{eq:F1}
F_1(\bx) = &\ 
\frac{1}{4}\sum_{1\le k \le N-1}|x-x_k|\, \t(x-x_k) \notag\\
&\ \quad\quad -\frac{Z}{2}\sum_{j=1}^{N-1} |x_j|\,\t(x_j) 
 + 
\frac{1}{4}\sum_{1\le j<k\le N-1}|x_j-x_k|\, \t(x_j-x_k).
\end{align} 
%
%
%
The function $e^F$ is only Lipschitz but, 
since the coefficients in \eqref{eq:jastrow} are bounded, we have 
$\phi \in \plainC{1,\mu}_{\tiny{\rm loc}}(\R^{3N})$ for 
all $\mu\in (0,1)$, see \cite[Proposition 1.5]{HOS2001}.  
Moreover, using the factorization 
$\psi = e^F\phi$ it was shown in \cite[Theorem 1.2]{HOS2001} that for all 
$\bx_0\in\R^{3N}$ and for any positive pair $r, R$ such that $r < R$, the bound holds:
\begin{align*}
\|\nabla\psi\|_{\plainL\infty(B(\bx_0, r))} \lesssim \|\psi\|_{\plainL\infty(B(\bx_0, R))}, 
\end{align*}
with a constant independent of $\psi$ and $\bx_0$.  
Slightly more general estimates 
were found in 
\cite[Proposition A.2]{FS2021}, see also \cite[Lemma 4.1]{HearnSob2023} for a shorter proof:
\begin{align}\label{eq:grphi}
\|\phi\|_{\plainL\infty(B(\bx_0, r))} + \|\nabla\phi\|_{\plainL\infty(B(\bx_0, r))} 
\lesssim \|\psi\|_{\plainL2(B(\bx_0, R))},
\end{align}
\begin{align}\label{eq:grpsi}
\|\psi\|_{\plainL\infty(B(\bx_0, r))} + \|\nabla\psi\|_{\plainL\infty(B(\bx_0, r))} 
\lesssim \|\psi\|_{\plainL2(B(\bx_0, R))}.
\end{align}
To improve the $\plainC{1, \mu}$-regularity of $\phi$, we 
introduce yet another Jastrow factor, 
\begin{align*}
F_3(\bx) = \frac{2-\pi}{12\pi} Z \sum_{1\le j < k\le N} 
x_j\cdot x_k \ln\big(|x_j|^2 + |x_k|^2\big) \t(x_j)\t(x_k). 
\end{align*}
According to \cite[Theorem 1.1, Remark 1.6]{FHOS2005}, the function $\tilde\phi$ defined as 
\begin{align*}
\tilde\phi(\bx) := e^{-F_3(\bx)} \, \phi(\bx) =  e^{-F(\bx) - F_3(\bx)} \psi(\bx),
\end{align*}
has $\plainC{1,1}(=\plainW{2, \infty})$-regularity:
 
\begin{prop}\label{prop:f3}
The function $\tilde\phi(\bx)$ 
belongs to $\plainW{2, \infty}(\R^{3N})$ and for any pair $r, R$, $0 < r < R$, and all 
$\bx_0\in\R^{3N}$ the bound holds 
\begin{align*}
\|\tilde\phi\|_{\plainW{2, \infty}(B(\bx_0, r))}\lesssim \| \psi\|_{\plainL\infty(B(\bx_0, R))}, 
\end{align*}
with a constant independent of $\psi$ and $\bx_0$. 
\end{prop}

%
The function $F_3$ belongs to $\plainW{2,p}(\R^{3N})$ for all 
$p < \infty$, but not for $p=\infty$. This 
is not sufficient for the $\plainW{2, \infty}$-smoothness of $\phi = e^{F_3}\tilde\phi$. 
Nevertheless, the following, slightly weaker fact is true (see  \cite{Hearnshaw2024}).  

\begin{lem} \label{lem:f3}
The derivative $\p_{x}^m \phi(\hat\bx, x)$, $|m|=2$, 
$m\in \mathbb N_0^3$,  
is globally bounded, and 
for any pair $r, R$, $0 < r < R$, and each $\bx_0\in\R^{3N}$ we have the estimate
\begin{align}\label{eq:nabla2phi}
\|\p_{x}^m\phi\|_{\plainL\infty(B(\bx_0, r))}\lesssim \|\psi\|_{\plainL2(B(\bx_0, R))}, 
\end{align}
with a constant independent of $\psi$ and $\bx_0$. 
\end{lem}

\begin{proof}
The derivatives of the function $F_3(\hat\bx, x)$ with respect to $x$ of first and second order 
are uniformly bounded. Therefore, by Proposition \ref{prop:f3},  
\begin{align*}
\|\p_{x}^m \phi\|_{\plainL\infty((B(\bx_0, r))} = \|\p_{x}^m\big(e^{F_3}\tilde\phi\big)\|_{\plainL{\infty}(B(\bx_0, r))}\lesssim \| \psi\|_{\plainL\infty(B(\bx_0, R))},
\end{align*}
and the bound \eqref{eq:nabla2phi} follows from \eqref{eq:grpsi}. 
\end{proof}

 As a step towards Theorem \ref{thm:varphider}, relying on the above estimates, we  
obtain bounds for higher order derivatives for $\phi$. For the first time 
these bounds were written out in \cite[Proposition 2.1]{Hearnshaw2024}.

\begin{thm}\label{thm:regphi} 
For all 
$\bx\notin\Sigma$, and for 
arbitrary $R >0$, we have the bound 
\begin{align}\label{eq:regphi}
|\p_{x}^m \,\phi(\bx)|
\lesssim \big(1+ \l(\bx)^{2-|m|}\big) \, \|\psi\|_{\plainL2(B(\bx, R))},
\quad m\in\mathbb N_0^{3},
\end{align}
with constants independent of $\psi$ and $\bx\notin\Sigma$. 
\end{thm}

The bound of the form \eqref{eq:regphi} with $\l(\bx)^{1-|m|}$ 
instead of 
$\l(\bx)^{2-|m|}$ was used in  \cite{FS2021} 
to derive \eqref{eq:FS}.

For the sake of completeness we provide a proof of Theorem \ref{thm:regphi}.

\subsection{Proof of Theorem \ref{thm:regphi}} 
First we quote a more general result that serves as a basis for the proof.

Let $u\in\plainW{1,2}$ be a weak solution of the equation 
\begin{align}\label{eq:mod}
-\Delta u + \ba\cdot\nabla u + b u = 0
\end{align}
in the ball $B(\bx_0, R\ell)\subset \R^{3N}$ 
with some $\bx_0\in \R^{3N}$, $0< \ell\le 1$, and 
a fixed $R>0$. Suppose that 
\begin{align}\label{eq:cd}
|\p_x^{m} \ba(\hat\bx, x)| + 
|\p_x^m b(\hat\bx, x)|\lesssim \ell^{-|m|}, \quad \bx\in B(\bx_0, R\ell),
\end{align}
for all $m\in\mathbb N_0^3$, with 
constants potentially depending on $m$, $R$, but not on $\bx_0$ and $\ell$. 
The next proposition that follows from  
\cite[Theorem 3.2]{HearnSob2023}, 
contains bounds for the  derivatives $\p_x^{m} u$  
with explicit dependence on the parameter $\ell\in (0, 1]$. 

\begin{prop}\label{prop:reg} 
Assume the conditions \eqref{eq:cd}, and let $u$ be a weak solution of the equation 
\eqref{eq:mod} in $B(\bx_0, R\ell)$.  
Then for all $m\in\mathbb N_0^3$ and all $r<R$ 
the derivatives $\p_x^m u$ belong to $\plainC{1, \mu}\big(\overline{B(\bx_0, r\ell)}\big)$ 
for all $\mu\in (0, 1)$. 
Furthermore, if $|m|\ge 2$, then   
\begin{align}\label{eq:d2ell}
\|\p_x^m u&\|_{\plainL\infty(B(\bx_0, r\ell))}\notag\\
\lesssim &\ \ell^{2-|m|}\,\big(
\max_{k: |k|=2} \|\p_x^k u\|_{\plainL\infty(B(\bx_0, R\ell))}
+ \|\nabla u\|_{\plainL\infty(B(\bx_0, R\ell))}
+  \|u\|_{\plainL\infty(B(\bx_0, R\ell))}\big).
\end{align}
The implicit constants in \eqref{eq:d2ell} are independent of $\bx_0$ and $\ell$, but may depend 
on the constants $r, R$, order $m$ of the derivative, and the 
constants in \eqref{eq:cd}. 
 \end{prop}

We apply Proposition \ref{prop:reg} to the equation \eqref{eq:jastrow} satisfied by the function 
$\phi$. 
In order to check that 
the coefficients of \eqref{eq:jastrow} satisfy 
the conditions \eqref{eq:cd}  we start with estimates of  
the Jastrow factor. 
Along with the distance function $\dc(\bx)$ defined in \eqref{eq:dq} introduce also 
\begin{align*}
\dc_1(\bx) = \min\{ \frac{1}{\sqrt2}\,|x-x_k|: 1\le k\le N-1\}, \quad \textup{and}\quad 
\l_1(\bx) = \min\{\dc_1(\bx), 1\}.
\end{align*}
Clearly, $\dc(\bx) = \min\{|x|, \dc_1(\bx)\}$, $\l(\bx) = \min\{|x|, \l_1(\bx)\}$.
\begin{lem}\label{lem:nablajas}
Let $F$, $F_1$ be as defined in \eqref{eq:jascut} and \eqref{eq:F1} 
respectively. Then for all 
$\bx\notin\Sigma$ and all 
$m\in\mathbb N_0^3$ we have
\begin{align}\label{eq:nablajas}
\big|\p_{x}^m \nabla F(\bx)\big| 
+ \big|\p_{x}^m |\nabla F(\bx)|^2\big|\lesssim &\ \dc(\bx)^{-|m|},\\[0.2cm]
\big|\p_{x}^m \nabla F_1(\bx)\big| 
+ \big|\p_{x}^m |\nabla F_1(\bx)|^2\big|\lesssim &\ \dc_1(\bx)^{-|m|}.\label{eq:nablajas0}
\end{align}
If $|m|\ge 1$, then  
\begin{align}\label{eq:ef}
\big|\p_{x}^m e^{F(\bx)}\big| \lesssim  \dc(\bx)^{1-|m|},\quad  
\big|\p_{x}^m e^{F_1(\bx)}\big| \lesssim  \dc_1(\bx)^{1-|m|}.
\end{align} 
\end{lem}

\begin{proof} 
Let $\t\in\plainC\infty_0(\R^3)$ be as in definition \eqref{eq:F0}, and let 
$g(x) = \nabla \big(|x| \t(x)\big)$. Then 
\begin{align*}
|\p_{x}^m g(x)| + |\p_{x}^m g(x-x_k)|\lesssim &\ \dc(\bx)^{-|m|},\notag\\
\quad  &\ |\p_{x}^m g(x-x_k)|\lesssim \dc_1(\bx)^{-|m|},\quad k = 1, 2, \dots, N-1,
\end{align*}
for all $\bx\notin \Sigma$. This entails the required bounds $\nabla F$ and $\nabla F_1$. 
For $|\nabla F|^2, |\nabla F_1|^2$ 
we use the Leibniz rule, which leads again to \eqref{eq:nablajas}, \eqref{eq:nablajas0}.

To prove the bounds \eqref{eq:ef} consider, for example, 
the derivative $\p_{x}^m e^{F_1}, |m|\ge 1$. The case of the function $F$ is done in the same way. 
Observe that $\p_{x}^m e^{F_1}, |m|\ge 1,$ 
can be written as a linear combination of finitely many terms of the form 
\begin{align*}
 \p_{x}^{k_1}F_1 \, 
 \p_{x}^{k_2}F_1 \cdots 
 \p_{x}^{k_s}F_1 \, e^{F_1},
\end{align*}
where $1 \le |k_j|$, $\sum_{j=1}^s |k_j| = |m|$, $1\le s\le |m|$.
By \eqref{eq:nablajas0} each such term can be estimated by 
\begin{align}\label{eq:dj}
\dc_1(\bx)^{(1-|k_1|) + (1-|k_2|) + \cdots (1-|k_s|) } 
= \dc_1(\bx)^{s - |m|}.
\end{align}
Due to the presence of the cut-off $\t$ in the definition \eqref{eq:F1}, the derivative 
$\p_{x}^m e^{F_1}$ is supported on the set where $\dc_1(\bx)\lesssim 1$. 
Therefore the right-hand side of \eqref{eq:dj} 
does not exceed $\dc_1(\bx)^{1-|m|}$, as required.
\end{proof}

Now we are in a position to prove Theorem \ref{thm:regphi}. 

\begin{proof}[Proof of Theorem \ref{thm:regphi}] 
For $|m| = 0, 1$ the bound \eqref{eq:regphi} holds because of \eqref{eq:grphi}. From now on assume that $|m|\ge 2$.

Further proof follows the idea of \cite[Sect. 2]{FS2021}. 
Let us fix 
$\bx_0\in \R^{3N}\setminus\Sigma$, and denote $\l = \l(\bx_0)$. 
Then for any $R_1\in (0, 1)$ we have   
$B(\bx_0, R_1\l)\subset \R^{3N}\setminus \Sigma$.  
Indeed, by \eqref{eq:lip},  
\begin{align*}
|\l(\bx) - \l|\le |\bx-\bx_0|<R_1\l,\quad \bx\in B(\bx_0, R_1\l),
\end{align*}
which implies that 
\begin{align}\label{eq:lip2}
0 <(1-R_1)\l \le \l(\bx)\le (1+R_1)\l,\quad \textup{for all}\quad \bx\in B(\bx_0, R_1\l),
\end{align} 
and hence proves the claim. 

By \eqref{eq:jastrow}, the function $\phi = e^{-F}\psi$ 
satisfies the equation \eqref{eq:mod} where 
\begin{align*}
\ba = -2\nabla F,\quad b = V- \Delta F - |\nabla F|^2 - E.
\end{align*}
Consider this equation in the ball $B(\bx_0, R_1\l)$ where  $R_1 < \min\{1, R\}$.
%
%
%
%
Since 
$B(\bx_0, R_1\l)\subset\R^{3N}\setminus\Sigma$, 
by virtue of \eqref{eq:Finf}, \eqref{eq:nablajas} and \eqref{eq:lip2}, 
the coefficients $\ba$ and $b$ 
satisfy the bound 
\begin{align*}
|\p_x^m\ba(\bx)| + |\p_x^m b(\bx)|\lesssim 1+ \dc(\bx)^{-|m|} 
\lesssim \l^{-|m|},\quad \textup{for all}\quad \bx\in B(\bx_0, R_1\l).
\end{align*}
Thus the condition \eqref{eq:cd} is fulfilled 
 with $\ell = \l\le 1$. By Proposition \ref{prop:reg},  
if $|m|\ge 2$, then  
\begin{align*}
|\p_x^{m}\phi(\bx_0)|\lesssim  \l^{2-|m|}\, \big(&
\max_{l: |l|=2} \|\p_x^l \phi\|_{\plainL\infty(B(\bx_0, R_1\l))}\notag\\
 &\ + 
\|\phi\|_{\plainL\infty(B(\bx_0, R_1\l))}
+ \|\nabla \phi\|_{\plainL\infty(B(\bx_0, R_1\l))}\big).
\end{align*}
Now the bound \eqref{eq:regphi} for $|m|\ge 2$ 
follows from \eqref{eq:grphi} and  \eqref{eq:nabla2phi} since $R_1\l\le R_1< R$. 
The proof is now complete. 
\end{proof}

\subsection{Proof of Theorem \ref{thm:varphider}}
For $|m| = 0, 1$ the required bound follows from \eqref{eq:grpsi}.  

Assume that \eqref{eq:van} holds and that $|m|\ge 2$.  
Since $e^{-F_0} \psi = e^{F_1}\phi$, 
we can write
\begin{align}\label{eq:f0f1}
\p_x^m\big(e^{-F_0} \psi\big) 
= 
\sum_{0\le k < m} {m\choose{k}} \p_x^k e^{F_1}\, \p_x^{m-k}\, \phi
+ \big(\p_x^m e^{F_1}\big)\, \phi
\end{align}
By \eqref{eq:regphi} and \eqref{eq:ef}, each term under the sum is bounded by 
\begin{align}\label{eq:summa}
\big(1+\l_1(\bx)^{1-|k|}\big) \, \big(1+ &\ \l(\bx)^{2-|m|+|k|}\big)
\,\|\psi\|_{\plainL2(B(\bx, R))}\notag\\
&\ \lesssim  
\big(1+\l(\bx)^{2-|m|}\big)\,\|\psi\|_{\plainL2(B(\bx, R))},\quad 0\le k < m,
\end{align}
where we have used that $\l(\bx)\le \l_1(\bx)\le 1$ and that $|k|+1\le |m|$.

To estimate the last term in \eqref{eq:f0f1} assume without loss of generality that 
$R \le 1$ and note that for $\l_1\ge {R}/{2}$ we have 
\begin{align*}
|\big(\p_x^m e^{F_1}\big)\, \phi|\lesssim \|\psi\|_{\plainL2(B(\bx, R))} 
\lesssim R^{|m|-2}\l_1(\bx)^{2-|m|} \, \|\psi\|_{\plainL2(B(\bx, R))}, 
\end{align*} 
by \eqref{eq:ef} and \eqref{eq:grphi}.
It remains to consider the case $\l_1(\bx) \le R/2$. Assume that 
$\l_1(\bx) = \frac{1}{\sqrt2}\, |x-x_j|$ with some $j = 1, 2, \dots, N-1$.
By \eqref{eq:van} we also have $\phi(\hat\bx, x_j) = 0$, so 
\begin{align*}
|\phi(\hat\bx, x)|  = &\ |\phi(\hat\bx, x) - \phi(\hat\bx, x_j)|\\
\le &\ \|\nabla \phi\|_{\plainL\infty(B(\bx, 3R/4))} |x-x_j|
\lesssim \sqrt 2\, \l_1(\bx)\, \|\psi\|_{\plainL2(B(\bx, R))},
\end{align*}
where we have used \eqref{eq:grphi}. Now, using also \eqref{eq:ef}, we obtain that 
\begin{align*}
|(\p_x^m e^{F_1(\bx)}) \phi(\bx)|
\lesssim &\ \l_1(\bx)^{2-|m|}\,  \|\psi\|_{\plainL2(B(\bx, R))}
\lesssim \l(\bx)^{2-|m|}\,  \|\psi\|_{\plainL2(B(\bx, R))}.
\end{align*}
Together with \eqref{eq:summa} the last bound leads to \eqref{eq:varphider}.
\qed
 
Note that without the assumption \eqref{eq:van}, 
repeating the above proof for the function $\psi$ instead of $e^{-F_0}\psi$, would give 
a proof of Proposition \ref{prop:FS}. 
%

\section{Compact operators} \label{sect:compact} 
Here we provide necessary information about compact operators. 
Most of the listed facts can be found in \cite[Chapter 11]{BS}. 
Let $\CH$ and $\mathcal G$ be separable Hilbert spaces.  
Let $A:\CH\to\mathcal G$ be a compact operator. 
If $\CH = \mathcal G$ and $A=A^*\ge 0$, then $\l_k(A)$, $k= 1, 2, \dots$, 
denote the positive eigenvalues of $A$ 
numbered in descending order counting multiplicity. 
For arbitrary spaces $\CH$, $\mathcal G$ and compact $A$, by $s_k(A) >0$, 
$k= 1, 2, \dots$, we denote the singular values of 
$A$ defined by $s_k(A)^2 = \l_k(A^*A) = \l_k(AA^*)$.   
We use the notation $\BS_p, p>0$ for the classical Schatten-von Neumann classes, 
i.e. the classes of compact operators  $A$ such that 
\begin{align*}
\| A\|_p := \bigg[\sum_{k=1}^\infty \, s_k(A)^p\bigg]^{\frac{1}{p}} < \infty. 
\end{align*}
Recall that $\|A\|_p$ defines a norm (resp. quasi-norm) if $p\ge 1$ (resp. $p <1$).
It is clear that $s_k(A)\le k^{-1/p} \|A\|_p$. The following estimate from 
\cite[Chapter 11, Theorem 4.7]{BS} will be useful.

\begin{prop}\label{prop:fr}
Let $A\in \BS_p, 0 < p <\infty$ and ${\rm rank}\, K\le n$. Then 
\begin{align*}
\sum_{k=n+1}^\infty s_k(A)^p \le \|A-K\|_p^p.
\end{align*}
In particular, 
\begin{align*}
s_{2n}(A)\le n^{-\frac{1}{p}}\,\| A - K\|_p. 
\end{align*}
\end{prop}

If $s_k(A)\lesssim k^{-1/p}, k = 1, 2, \dots$, 
with some $p >0$, then we say that $A\in \BS_{p, \infty}$ and denote
\begin{align*}
\| A\|_{p, \infty} = \sup_k s_k(A) k^{\frac{1}{p}},
\end{align*}
so that 
\begin{align}\label{eq:star}
\|A^*A\|_{p/2, \infty} = \|A A^*\|_{p/2, \infty} = \|A\|_{p, \infty}^2. 
\end{align}
Clearly, $\BS_p\subset \BS_{p, \infty}$ and $\|A\|_{p, \infty}\le \|A\|_p$.
The class $\BS_{p, \infty}$ is a complete linear space with the quasi-norm $\|A\|_{p, \infty}$. 
For all $p>0$ the functional $\|A\|_{p, \infty}$ 
satisfies the following ``triangle" inequality for  
operators $A_1, A_2\in\BS_{p, \infty}$:
\begin{align}\label{eq:triangle}
\|A_1+A_2\|_{\scalet{p}, \scalel{\infty}}^{{\frac{\scalel{p}}{\scalet{p+1}}}}
\le \|A_1\|_{\scalet{p}, \scalel{\infty}}^{{\frac{\scalel{p}}{\scalet{p+1}}}}
+ \|A_2\|_{\scalet{p, \infty}}^{{\frac{\scalel{p}}{\scalet{p+1}}}}.
\end{align}  
This inequality allows one to estimate quasi-norms of ``block-vector" operators. 
Let $A_j\in\BS_{p, \infty}$ 
be a finite collection of compact operators. Define the operator $\BA:\CH\to\oplus_j \mathcal G$ by  
$\BA = \{A_j\}_{j}$.
Since $\BA^* \BA = \sum_j A_j^*A_j$, 
and $A_j^*A_j\in\BS_{q, \infty}$, $q = p/2$, 
by \eqref{eq:triangle} we have 
\begin{align}\label{eq:blockvec}
\|\BA\|_{p, \infty}^{{\frac{\scalel{2p}}{\scalet{p+2}}}} 
= \|\BA^*\BA\|_{q, \infty}^{{\frac{\scalel{q}}{\scalet{q+1}}}}
\le  \sum_j\|A_j^*A_j\|_{q, \infty}^{{\frac{\scalel{q}}{\scalet{q+1}}}} = 
\sum_j\|A_j\|_{p, \infty}^{{\frac{\scalel{2p}}{\scalet{p+2}}}}.
\end{align}
Consequently, in order to estimate the singular values of $\BA$ it suffices to 
estimate those of its components $A_j$. We use this fact throughout the paper. 
%
 
For $p\in (0, 1)$ it is often more convenient to use the following ``triangle" inequality for  
arbitrarily many operators $A_j\in\BS_{p, \infty}$, $j = 1, 2, \dots$:
\begin{align}\label{eq:ptriangle}
\big\|\sum_{j} A_j\big\|_{p, \infty}^p\le (1-p)^{-1}
\sum_j \|A_j\|_{p, \infty}^p,
\end{align}
see 
\cite[Lemmata 7.5, 7.6]{AJPR2002}, \cite[\S 1]{BS1977} and references therein. Under 
the additional assumption
\begin{align}\label{eq:orthog}
A_k A_j^* = 0\quad \textup{or}\quad A_k^*A_j = 0,\quad j\not = k, 
\end{align}   
the inequality of the form \eqref{eq:ptriangle} extends to all $p\in (0, 2)$: 

\begin{prop}\cite[Lemma 1.1]{BS1977}\label{prop:orthog}
Let $p\in (0, 2)$ and let $A_j\in \BS_{p, \infty}$, $j= 1, 2, \dots$, 
be a family of operators satisfying 
\eqref{eq:orthog}. 
Then for the operator $A = \sum_j A_j$ we have the inequality
\begin{align}\label{eq:pnorm}
\|A\|_{\scalet{p}, \scalel{\infty}}^{\scalet{p}}\le 
\frac{2}{2-p}\,
\sum_j \|A_j\|_{\scalet{p}, \scalel{\infty}}^{\scalet{p}}. 
\end{align} 
\end{prop}

\begin{proof} 
It suffices to conduct the proof under the first condition in 
\eqref{eq:orthog} only, so that  
\begin{align*}
A A^* = \sum_j A_j A_j^*. 
\end{align*}
Since $A_j A_j^*\in\BS_{q, \infty}, q = p/2<1$, 
and $\|A_j A_j^* \|_{q, \infty} = \|A_j\|_{p, \infty}^2$, 
the inequality \eqref{eq:ptriangle} leads to the bound 
\begin{align*}
\|A\|_{p, \infty}^p = \|A A^*\|_{q, \infty}^q \le (1-q)^{-1} 
\sum_j \|A_j A_j^*\|_{q, \infty}^q
= (1-p/2)^{-1}\sum_j \|A_j\|_{p, \infty}^p. 
\end{align*}
This inequality coincides with \eqref{eq:pnorm}.
\end{proof}

\section{Singular values of integral operators}\label{sect:int}
 
\subsection{Main estimate} 
The membership of integrals operators in various classes of compact operators, including 
$\BS_p$ and $\BS_{p, \infty}$, has been comprehensively investigated by   
 M.S. Birman and M.Z. Solomyak, in \cite{BS1977}. In this monograph 
such properties are formulated in terms of functional classes to which the integral kernels 
belong. 
Although the results that we need can be derived from \cite{BS1977}, as in 
\cite{Sobolev2022a, Sobolev2022b},
the integral kernels that we consider admit a more elementary and direct approach, 
which we describe below.  

We are interested in integral operators whose kernels imitate the structure 
of the eigenfunction $\psi$. More precisely, we consider integral operators 
$\iop(T)$ with scalar kernels $T(t, x)$, $t\in \R^l, x\in\R^d$, $l\ge1, d\ge 1$, 
satisfying the following condition. 
Let $z_k: \R^l\mapsto \R^d$,\, 
$k = 1, 2, \dots, N$, be some functions of $t\in\R^l$, 
and let $X(t) = \{z_1(t), z_2(t), \dots, z_N(t)\}\subset \R^d$.
Assume that 
for some $\a > -d$, for a.e. $t\in\R^l$ the function $T(t, x)$ is such that 
\begin{align*}
T(t, \,\cdot\,)\in\plainC\varkappa(\R^d\setminus X(t)),\quad \varkappa = \lfloor\a\rfloor + d + 1, 
\end{align*}
and that it satisfies the bounds 
   \begin{align}\label{eq:kernelexp}
|\p_x^j T(t, x)| \le A(t, x) 
\bigg(1+\sum_{k=1}^N \big(1\wedge |x-z_k(t)|\big)^{\a-|j|}\bigg),\quad  
|j| \le \varkappa,
\end{align}    
with some non-negative function $A(t, x)$. Introduce the notation 
\begin{align}\label{eq:cubes}
\CC^{(r)} = [-r/2, r/2)^d, r>0;\quad \CC^{(r)}_n = \CC^{(r)} + n,\, n\in\Z^d.
\end{align}
If $r=1$ then we omit the superscript ``$(r)$" and write $\CC = \CC^{(1)}$, 
$\CC_n = \CC+n$, $n\in\Z^d$. All constants in the subsequent inequalities may depend 
on the dimensions $l$, $d$, numbers $\a$, $r$ and $N$, but they will not depend on the functions $z_k(t)$ and $A(t, x)$. 
Our main objective in this section is the following theorem.

\begin{thm}\label{thm:kg} 
Let the kernel $T$ be as specified above, and let 
$a\in\plainL{\infty}_{\textup{\tiny loc}}(\R^d)$. Suppose  
that for each $n\in \Z^d$ the function 
$A_n(t):=\|A(t,\,\cdot\,)\|_{\plainL\infty(\CC^{(2\pi)}_n)}$ is such that 
$A_n \in \plainL2(\R^l)$.
If $\a > -d/2$, then the operator $\iop(T) a$ 
belongs to $\BS_{q, \infty}$ with $q^{-1} = 1+\a d^{-1}$, and 
\begin{align}\label{eq:pivg}
\| \iop(T) a\|_{q, \infty}\lesssim 
\bigg[\sum_{n\in\Z^d}
\| A_n\|_{\plainL2(\R^l)}^q\,\|a\|_{\plainL{\infty}(\CC_n)}^q
\bigg]^{\frac{1}{q}},
\end{align}
under the assumption that the right-hand side of the above inequality is finite. 
The implicit constant in \eqref{eq:pivg} 
does not depend on the functions $z_k$, $A$, $a$.
\end{thm}

\begin{rem}\label{rem:vector}
\begin{enumerate}
\item \label{item:vector}
Although we assume that the kernel $T(t,x)$ in the above theorem is scalar, the estimate 
\eqref{eq:pivg} also holds for vector-valued kernels if the conditions of Theorem 
\ref{thm:kg} are satisfied for each component. To see this it suffices to use the bound \eqref{eq:blockvec}.
\item 
The bound \eqref{eq:pivg} is similar 
to the one obtained in \cite[Theorem 3.4]{Sobolev2022a} with the help of the general 
estimates from \cite{BS1977}. 
In contrast to 
\eqref{eq:pivg}, the bound in \cite{Sobolev2022a} contained 
$\plainL{p}$-norms of $a$ with some $p = p(\a)\ge 2$, instead of 
the $\plainL\infty$-norms of $a$.  
\end{enumerate}
\end{rem}

The proof relies on two lemmata below.

\subsection{Fourier transform} 
First we examine the Fourier transform of functions $u = u(x)$ satisfying a bound similar to \eqref{eq:kernelexp}. 

\begin{lem}\label{lem:fourier} 
Let $X = \{a_1, a_2,\dots, a_N\}$ be a finite collection of points in $\R^d$. 
Let the function $u$ be such that for some $\a > - d$ it satifies  
\begin{align*}
u\in\plainC\varkappa(\R^d\setminus X), \quad \varkappa  = \lfloor\a\rfloor + d + 1, 
\end{align*} 
and the following bounds hold:  
\begin{align}\label{eq:sing}
|\p_x^m u(x)|\lesssim \1_{\CC^{(r)}}(x) \big(1+ \sum_{k=1}^N (|x-a_k|\wedge 1)^{\a-|m|}\big),
\end{align}
with some $r>0$ and   
for all $x\notin X$ and all $|m|\le \varkappa$. 
Then 
\begin{align}\label{eq:fourier}
|\hat u(\xi)|\lesssim \lu\xi\ru^{-\a-d}, \quad \lu \xi\ru = \sqrt{1+|\xi|^2},
\end{align}
where the implicit constant does not depend on the set $X$, 
but may depend on $r$ and $N$. 
\end{lem} 

\begin{proof} 
Clearly, it suffices to prove \eqref{eq:fourier} for $|\xi|\ge 1$.

First we prove the required estimate for $\a\le 0$, so that 
\eqref{eq:sing} amounts to 
\begin{align}\label{eq:sing1}
|\p_x^m u(x)|\lesssim \1_{\CC^{(r)}}(x) 
\big(1+\sum_{k=1}^N |x-a_k|^{\a-|m|}\big),\quad |m|\le \varkappa.
\end{align}
For each $\varepsilon\in (0, 1]$ 
introduce a smooth partition of unity $\eta(x) = \eta(x, \varepsilon)$,\, 
$\z(x) = \z(x, \varepsilon)$, such that $\eta+\z = 1$, $0\le \eta\le 1$, and 
\begin{align*}
\eta(x; \varepsilon) = 0,\quad & \textup{if}\quad \dist(x; X)\ge \varepsilon;\\
\z(x; \varepsilon) = 0,\quad & \textup{if}\quad \dist(x; X) <  \frac{\varepsilon}{2},
\end{align*} 
and 
\begin{align}\label{eq:tz}
|\p_x^m \eta(x)|\lesssim \varepsilon^{-|m|},\quad 
|\p_x^m \z(x)|\lesssim \varepsilon^{-|m|},\quad m\in \mathbb N_0^d,
\end{align}
where the implicit constants do not depend on the set $X$. 
Represent $\hat u(\xi) = I_0(\xi, \varepsilon) + I_1(\xi, \varepsilon)$ with 
\begin{align*}
I_0(\xi; \varepsilon) = \frac{1}{(2\pi)^{\frac{d}{2}}}\int e^{-ix\cdot\xi}\, u(x)\, 
\eta(x; \varepsilon)\, dx\quad \textup{and}\ 
I_1(\xi; \varepsilon) = \frac{1}{(2\pi)^{\frac{d}{2}}} 
\int e^{-ix\cdot\xi}\, u(x)\, \z(x; \varepsilon)\, dx,
\end{align*}
and estimate $I_0$ and $I_1$ separately. 
Let 
\begin{align*}
\Om_k = \{x\in\CC^{(r)}: \dist(x; X) = |x-a_k|\},\quad k = 1, 2, \dots, N,
\end{align*}
so $\CC^{(r)} = \cup_k\Om_k$. Then, by \eqref{eq:sing1},
\begin{align}\label{eq:i0}
|I_0(\xi, \varepsilon)|\lesssim &\ 
\sum_{k=1}^N\, \int\limits_{\Om_k\cap B(a_k, \varepsilon)} |x-a_k|^\a\, dx\notag\\
\lesssim &\ \sum_{k=1}^N\, \int\limits_{|x-a_k| \le \varepsilon} |x-a_k|^\a\, dx
 \lesssim \varepsilon^{\a+d}.
\end{align}
To estimate $I_1$ 
%
%
integrate by parts $\varkappa$ times 
using the relation $|\xi|^{-2}\, i\xi\cdot\nabla_x\, e^{-ix\cdot\xi} = e^{-ix\cdot\xi}$:
\begin{align*}
I_1 =  \frac{1}{(2\pi)^{\frac{d}{2}}}
\int e^{-ix\cdot\xi} (-|\xi|^{-2}\, i\xi\cdot\nabla_x)^\varkappa\, 
\big(u(x)\, \z(x; \varepsilon)\big)\, dx.  
\end{align*}
It is clear that 
\begin{align*}
|I_1 |\lesssim 
|\xi|^{-\varkappa}\,\sum_{p, s: |p|+|s| = \varkappa}\, 
\int\, |\p_x^s u(x)|\, |\p_x^p \z(x; \varepsilon)|\, dx.
\end{align*}
First estimate the integral with $p = 0$:
\begin{align*}
\int\, |\p_x^s u(x)|\, dx
\lesssim &\  \sum_{k=1}^N\, 
\int_{\Om_k\setminus B(a_k, \varepsilon/2)} (1+|x-a_k|^{\a-\varkappa})\, dx \\
\lesssim &\  \sum_{k=1}^N\, 
\int\limits_{x\in \CC^{(r)}, |x-a_k|>\varepsilon/2} (1+|x-a_k|^{\a-\varkappa})\, dx
\lesssim \varepsilon^{\a-\varkappa+d},
\end{align*}
where we have used the fact that $\a-\varkappa<-d$. Suppose that $|p| \ge 1$. By the definition 
of $\z$, the derivative $\p_x^p\z$ is supported on the set where 
$\varepsilon/2\le  \dist(x, X)\le \varepsilon$. 
Therefore, by \eqref{eq:tz}, 
\begin{align*}
\int\, |\p_x^s u(x)|\, |\p_x^p \z(x; r)|\, dx
\lesssim &\ \sum_{k=1}^N\, \int_{(\Om_k\cap B(a_k, \varepsilon))\setminus B(a_k, \varepsilon/2)}
(1+|x-a_k|^{\a-|s|})\, \varepsilon^{-|p|} dx \\
\lesssim &\ (1+\varepsilon^{\a-|s|}) \varepsilon^{-|p|}  \varepsilon^d
\lesssim  \varepsilon^{\a-\varkappa +d}.
\end{align*}
Consequently, 
\begin{align*}
|I_1(\xi, \varepsilon) |\lesssim |\xi|^{-\varkappa} \varepsilon^{\a-\varkappa+d}.
\end{align*}
Combining this bound with \eqref{eq:i0} we obtain:
\begin{align*}
|\hat u(\xi)|\lesssim \big( \varepsilon^{\a+d} + |\xi|^{-\varkappa} 
\varepsilon^{\a-\varkappa+d}\big). 
\end{align*}
Taking $\varepsilon = |\xi|^{-1}$ we conclude the proof of the claimed bound 
\eqref{eq:fourier} for $\a\le 0$. 

For arbitrary $\a >0$, let 
$m = \lceil \a\rceil < \varkappa$, so that $-1<\a-m\le 0$. 
Rewrite the Fourier transform of $u$ integrating by parts $m$ times:
\begin{align*}
\hat u(\xi) = \frac{1}{(2\pi)^{\frac{d}{2}}}
\int e^{-ix\cdot\xi} (-|\xi|^{-2}\, i\xi\cdot\nabla_x)^m\,  u(x) \, dx,
\end{align*}
so that
\begin{align*}
|\hat u(\xi)|\lesssim |\xi|^{-m} \sum_{|n|=m}
%
%
|\widehat{(\p^n u)}(\xi)|.
\end{align*}
Since $\b:=\a-m > -1$, the right-hand side is finite for all $\xi\in\R^d$.  
The function 
%
%
$v = \p^n_x u, |n|=m$, satisfies $v\in\plainC{\om}(\R^d\setminus X)$ 
with $\om = \lfloor\b\rfloor + d+1$, and the bounds hold:
\begin{align*}
|\p_x^s v(x)|\lesssim  (1+ \sum_k |x-a_k|^{\b - s}),\quad |s|\le \om.
\end{align*}
Since $\b\le 0$, it follows from the first part of the proof, 
that 
\begin{align*}
|\hat u(\xi)|\lesssim  |\xi|^{-m} \, |\xi|^{-\b -d} =  |\xi|^{-\a-d},
\end{align*}
which gives \eqref{eq:fourier} and completes the proof.
\end{proof}

\subsection{Operators on the cube} 
The next step towards the proof of Theorem \ref{thm:kg} is the $s$-value estimate 
for integral operators on the cube. 

\begin{lem}\label{lem:cubepi}
Let the kernel $T$ be as in Theorem 
\ref{thm:kg}, and assume that $A(t, x) = A(t, x)\1_{\CC^{(2\pi)}}(x)$. 
Suppose that  
$\tilde A(t):=\|A(t,\,\cdot\,)\|_{\plainL\infty(\CC^{(2\pi)})}$ is such that 
$\tilde A\in\plainL2(\R^l)$.
If $\a > -d/2$, then 
the operator $ \iop(T)$ 
belongs to $\BS_{q, \infty}$, with $q^{-1} = 1+\a d^{-1}$, and 
\begin{align}\label{eq:piv}
\|\iop(T)\|_{q, \infty}\lesssim 
\|\tilde A \|_{\plainL2(\R^l)}.
\end{align} 
The implicit constant in \eqref{eq:piv} does not depend on the functions $z_k$.
\end{lem}

\begin{proof} 
Since $T(t, x) = T(t, x) \1_{\CC^{(2\pi)}}(x)$, we consider 
$\iop(T)$ as an operator acting on $\plainL2(\CC^{(2\pi)})$ . 
Denote
\begin{align*}
\hat T_\nu(t) = \frac{1}{(2\pi)^{\frac{d}{2}}} \int_{\CC^{(2\pi)}} \, 
e^{-i\nu\cdot x}\, T(t, x) \, dx,\quad \nu\in\Z^d.  
\end{align*}
By Lemma \ref{lem:fourier}, $|\hat T_\nu(t)|\lesssim \tilde A(t) \lu \nu\ru^{-\a-d}$, 
so that $\hat T_\nu\in \plainL2(\R^l)$ and  for a.e. $t\in \R^l$ the series 
\begin{align*}
T(t, x) = \frac{1}{(2\pi)^{\frac{d}{2}}} \sum_{\nu\in\mathbb Z^d} \, 
e^{i\nu\cdot x}\, \hat T_\nu(t),
\end{align*}
converges in $\plainL2(\CC^{(2\pi)})$. 
For a fixed $M\ge 1$ the rank of the operator 
$\iop(T_M):\plainL2(\CC^{(2\pi)})\mapsto \plainL2(\R^l)$   
with kernel 
\begin{align*}
T_M(t, x) = \frac{1}{(2\pi)^{\frac{d}{2}}} \sum_{|\nu|\le M} \, 
e^{i\nu\cdot x}\, \hat T_\nu(t),
\end{align*}
does not exceed  
\begin{align*}
m = m(M):= \#\{\nu\in \Z^d: |\nu| \le M\}
\asymp M^d.
\end{align*}
Therefore, by Proposition \ref{prop:fr},
\begin{align}\label{eq:finiterank}
s_{2m}(\iop(T))\le m^{-\frac{1}{2}}\, \| \iop(T) - \iop(T_M)\|_2. 
\end{align}
Estimate the Hilbert-Schmidt norm on the right-hand side:
\begin{align*}
\int_{\R^l}\, \int_{\CC^{(2\pi)}}|T(t, x) - & T_M(t, x)|^2 \, dx dt
=  \sum_{|\nu|>M} \int_{\R^l}\, |\hat T_\nu(t)|^2 \, dt\\
\lesssim  &\ \| \tilde A\|_{\plainL2}^2\, \sum_{|\nu|>M} \, \lu \nu\ru^{-2(\a+d)}
\lesssim  M^{-2\a - d}\| \tilde A\|_{\plainL2}^2\lesssim m^{-\frac{2}{q}+ 1}
\| \tilde A\|_{\plainL2}^2.
\end{align*}
By \eqref{eq:finiterank}, 
\begin{align*}
s_{2m}(\iop(T))\lesssim 
m^{-\frac{1}{2}} \bigg(
m^{-\frac{1}{q} + \frac{1}{2}} \ \|\tilde A\|_{\plainL2}\bigg) = 
m^{-\frac{1}{q}}\, \|\tilde A\|_{\plainL2}.
\end{align*}
Thus for any $k: 2m(M)\le k < 2m(M+1)$, we have 
\begin{align*}
s_k(\iop(T))\le s_{2m(M)}(\iop(T))\lesssim \bigg(\frac{k}{m(M)}\bigg)^{\frac{1}{q}}
k^{-\frac{1}{q}} \|\tilde A\|_{\plainL2}\le \bigg(\frac{2 m(M+1)}{m(M)}\bigg)^{\frac{1}{q}}
k^{-\frac{1}{q}} \|\tilde A\|_{\plainL2}.
\end{align*}
Since $m(M+1)\lesssim m(M)$ with a constant independent of $M$, the singular values satisfy the 
bound $s_k(\iop(T))\lesssim k^{-\frac{1}{q}} \|\tilde A\|_{\plainL2}$ 
 for all $k\ge 2m(1)$. The same bound also holds for all $k: 1\le k\le 2 m(1)$ due to the straighforward estimate $s_k(\iop(T))\le \|\iop(T)\|_2 \lesssim \|\tilde A\|_{\plainL2}$. 
\end{proof}

Now we are ready to prove Theorem \ref{thm:kg}.
 
\subsection{Proof of Theorem \ref{thm:kg}}
Without loss of generality assume that $a$ is constant on each cube $\CC_n=\CC^{(1)}_n$, and write 
$a_n = a(x)\1_n(x)$ where $\1_n$ is the indicator of $\CC_n$.
Consider the 
kernel $W_n(t, x) = T(t, x) a_n \1_n(x)$, so that
\begin{align*}
\iop(T)\,a = \sum_{n\in\Z^d} \, \iop( W_n).
\end{align*}
Let $\eta\in \plainC\infty_0(\CC^{(2\pi)})$  
be a function such that $\eta(x) = 1$ if $x\in \CC$, and denote 
\[
\widetilde W_n(t, x) = T(t, x)\eta_n(x) a_n,\quad \eta_n(x) = \eta(x-n).
\]
 Clearly,  
$s_k\big(\iop(W_n)\big)\le s_k\big(\iop(\widetilde W_n)\big),\  k = 1, 2, \dots$. 
The kernel $T(t, x)\eta_n(x)$ satisfies \eqref{eq:kernelexp} with the function 
$A(t, x)\1_{\CC^{(2\pi)}_n}(x)$. Thus, 
by Lemma \ref{lem:cubepi}, the operator $\widetilde W_n$ belongs to 
$\BS_{q, \infty}$ with $1/q = 1+ \a/d$, and 
\begin{align*}
\| \iop(W_n)\|_{q, \infty} 
\le &\  \| \iop(\widetilde W_n)\|_{q, \infty}\notag\\[0.2cm]
\lesssim &\ \|A_n\|_{\plainL2(\R^l)}
 \, |a_n|,
 \quad A_n(t) = \|A(t, \,\cdot\,)\|_{\plainL\infty(\CC^{(2\pi)}_n)}, 
\end{align*} 
with a constant independent of $n$. 
Now observe that $\iop(W_k)\, (\iop W_j)^* = 0$ for $k\not = j$. Furthermore, 
since $\a > -d/2$ we have $q < 2$, and hence by Proposition \ref{prop:orthog},
\begin{align*}
\| \iop(T) a\|_{q, \infty}^q\le &\ 2(2-q)^{-1} \sum_{n\in\mathbb Z^d} 
\|\iop(W_n)\|_{q, \infty}^q\\
\lesssim &\  \sum_{n\in\Z^d}\|A_n\|_{\plainL2(\R^l)}^q
 \, \|a\|_{\plainL{\infty}(\CC_n)}^q.
\end{align*}
This completes the proof of \eqref{eq:pivg}.
\qed

\section{Proof of Theorems \ref{thm:main} and \ref{thm:main0}}\label{sect:proof}

We rely on the factorization \eqref{eq:factor} with operators $\Psi$ 
and $\SV$ that are defined in \eqref{eq:psiv}. First we study $\Psi$ and $\SV$.

\subsection{The weighted operators $\Psi$ and $\SV$}  
For methodological purposes (e.g. for the study of the eigenvalue asymptotics, 
see \cite{Sobolev2022, Sobolev2022a} )
we obtain estimates for the operators $b \Psi a$ and $b \SV a$ 
with weights $a$ and $b$. 
In what follows we use the notation $ \CC^{(r)}_n, n\in\Z^d, r>0,$ 
that was introduced in \eqref{eq:cubes}. For $f\in\plainL2_{\rm loc}(\R^{3N-3})$ define 
the ``weighted" one-particle density (cf. \eqref{eq:dens}) 
\begin{align*}
\rho[f](x) =  
\int_{\R^{3N-3}} \, |f(\hat\bx)|^2 \, |\psi(\hat\bx, x)|^2   \, d\hat\bx,
\end{align*}
and introduce 
the mean value 
\begin{align*}
\tilde f_R(\hat\bx) = \bigg[ \frac{1}{|B(\hat{\mathbf 0}, R)|} 
\int_{B(\hat\bx, R)} |f(\hat\by)|^2\, d\hat\by\bigg]^{\frac{1}{2}},
\end{align*}
where $B(\hat\bx, R)\subset\R^{3N-3}$ is the ball of radius $R>0$ 
centered at $\hat{\bx}$.
Finally, for functions $f\in\plainL2_{\rm loc}(\R^{3N-3})$ and 
$a\in \plainL\infty_{\rm loc}(\R^3)$ define 
\begin{align*}
M_q(a,f) = \bigg[\sum_{n\in\Z^d}
\|\,\rho[f]\,\|_{\plainL1(\CC^{(4\pi)}_n)}^{\frac{q}{2}}\,\|a\|_{\plainL{\infty}(\CC_n)}^{q}
\bigg]^{\frac{1}{q}},\quad q >0.
\end{align*}
Now we are in a position to state $s$-value bounds for the weighted operators 
$b\, \Psi\, a$ and $b\, \SV\, a$. 
We assume that the coefficients $M_q$ on the right-hand sides of the bounds below, are finite. 

\begin{thm}\label{thm:psifull} 
Assume that $b\in\plainL2_{\rm loc}(\R^{3N-3})$ and 
$a\in \plainL\infty_{\rm loc}(\R^3)$. Let $R>0$ be arbitrary. 
Then  $ b \Psi a\in \BS_{3/4, \infty}$, $ b {\sf V} a\in \BS_{1, \infty}$, and 
\begin{align}\label{eq:psi}
\|b\,\Psi\, a\|_{3/4, \infty}\lesssim 
&\ M_{3/4}(a, \tilde b_R),\\ 
\|b\, {\sf V}\, a\|_{1, \infty}\lesssim 
&\ M_{1}(a, \tilde b_R).\label{eq:v}
\end{align}
If \eqref{eq:van} is satisfied, then $ b \Psi a\in \BS_{3/5, \infty}$, 
$b {\sf V} a\in \BS_{3/4, \infty}$, and 
\begin{align}\label{eq:psi0}
\|b\,\Psi\, a\|_{3/5, \infty}\lesssim 
&\ M_{3/5}(a, \tilde b_R),\\
\|b\, {\sf V}\, a\|_{3/4, \infty}\lesssim 
&\ M_{3/4}(a, \tilde b_R).\label{eq:v0}
\end{align}
The constants in the above bounds do not depend on $\psi, b, a$ but may depend on $R$ and $N$.
\end{thm}

\begin{proof}
We use Theorem \ref{thm:kg} with $d=3$ and $l = 3N-3$. 
By \eqref{eq:FS}, the kernel $T(\hat\bx, x) = b(\hat\bx)\psi(\hat\bx, x)$ 
satisfies the bound \eqref{eq:kernelexp} 
with $\a = 1$, $z_j(x) = x_j, j = 1, 2, \dots, N-1$, 
and 
\begin{align}\label{eq:newa}
A(\hat\bx, x) = |b(\hat\bx)|\|\psi\|_{\plainL2(B(\bx, R))},
\end{align}
$\bx = (\hat\bx, x)$. Without loss of generality assume that $R\le \pi$. 
Therefore 
\begin{align*}
A_n(\hat\bx) = \|A(\hat\bx,\,\cdot\,)\|_{\plainL\infty(\CC^{(2\pi)}_n)}\le 
|b(\hat\bx)| 
\|\psi\|_{\plainL2(B_n(\hat\bx))},\ \quad B_n(\hat\bx) = B(\hat\bx, R)\times \CC^{(4\pi)}_{n}.
\end{align*}
As a consequence, 
\begin{align*}
\|A_n\|_{\plainL2(\R^{3N-3})}^2
\le &\ \int_{\R^{3N-3}} \,  |b(\hat\bx)|^2   
\int_{|\hat\by-\hat\bx|<R}   \int_{\CC^{(4\pi)}_n} |\psi(\hat\by, x)|^2\,   
dx\, d\hat\by  d\hat\bx\\
= &\ \int_{\CC^{(4\pi)}_n}\int_{\R^{3N-3}} \,     
\bigg(\int_{|\hat\by-\hat\bx|<R} |b(\hat\bx)|^2 d\hat\bx\bigg) \, |\psi(\hat\by, x)|^2\,
d\hat\by\, dx\\
\lesssim  &\  \int_{\CC_n^{(4\pi)}}  \, \rho[\tilde b_R](x)  \, dx.
\end{align*}
Now the bound \eqref{eq:psi} follows from \eqref{eq:pivg} where $1/q = 1+ \a/3 = 4/3$. 

The kernel $T(\hat\bx, x) = b(\hat\bx)\nabla\psi(\hat\bx, x)$ 
satisfies \eqref{eq:kernelexp} with the function \eqref{eq:newa} and with $\a = 0$. 
A reference to \eqref{eq:pivg} and 
Remark \ref{rem:vector}\eqref{item:vector} leads to \eqref{eq:v}.

Assume the condition \eqref{eq:van}. Rewrite 
$b\Psi a$ using the function $F_0$ defined in \eqref{eq:F0}:
\begin{align*}
b\,\Psi\, a = \iop(b e^{-F_0} \psi)\, e^{F_0} a.
\end{align*}
By Theorem \ref{thm:varphider}, 
the kernel $be^{-F_0} \psi$ satisfies \eqref{eq:kernelexp} with 
$\a = 2$ and the same function \eqref{eq:newa}. 
Apply \eqref{eq:pivg} with $e^{F_0}a$ instead of $a$. Since  $e^{F_0}\lesssim 1$, 
\eqref{eq:pivg} implies \eqref{eq:psi0}. 
For the proof of \eqref{eq:v0} represent $\nabla_x\psi$ as follows:
\begin{align*}
\nabla_x\psi = e^{F_0} \nabla_x \big(e^{-F_0}\psi\big) + \psi\nabla_x F_0,  
\end{align*}
so that 
\begin{align}\label{eq:np}
b\, \SV\, a = \iop(b\, \nabla_x\psi)\, a 
= \iop\big(b \, \nabla_x (\psi e^{-F_0})\big)\, e^{F_0} a 
+ \iop\big( b\, \psi\big)\, a \nabla_x F_0.
\end{align}
By \eqref{eq:varphider},  
the kernel $b \nabla_x(\psi e^{-F_0})$ satisfies \eqref{eq:kernelexp} with $\a = 1$, and the same function $A(\hat\bx, x)$ as before. 
Using \eqref{eq:pivg} again, we obtain 
\eqref{eq:v0} for the first term in \eqref{eq:np}. Since $|\nabla F_0|\lesssim 1$, the bound 
\eqref{eq:v0} for the second term in \eqref{eq:np} follows from \eqref{eq:psi}. 
Using 
\eqref{eq:triangle} or \eqref{eq:pnorm} for the two terms we arrive at \eqref{eq:v0}.
\end{proof}

\subsection{Proof of Theorems \ref{thm:main} and \ref{thm:main0}}
Let $a = 1$ and $b = 1$, so that $\tilde b_R = 1$. 
Thus, by \eqref{eq:star}, it follows from \eqref{eq:factor} and 
\eqref{eq:psi}, \eqref{eq:v} that 
\begin{align*}
\|\iop(\g)\|_{3/8, \infty}
= &\ \|\Psi\|_{3/4, \infty}^2\lesssim M_{3/4}(1,1)^2,\\
\|\iop(\tau)\|_{2, \infty}
= &\ \|\SV\|_{1, \infty}^2 \lesssim M_{1}(1,1)^2. 
\end{align*}
Since 
\begin{align*}
M_q(1, 1)^2 = \bigg[\sum_{n\in\Z^d}
\|\,\rho\,\|_{\plainL1(\CC^{(4\pi)}_n)}^{\frac{q}{2}}\, \bigg]^{\frac{2}{q}}
\lesssim \bigg[\sum_{n\in\Z^d}
\|\,\rho\,\|_{\plainL1(\CC_n)}^{\frac{q}{2}}\, \bigg]^{\frac{2}{q}} 
= \4\, \rho\, \4\,_{{q}/2} , 
\end{align*}
the above bounds imply Theorem \ref{thm:main}.
The bounds \eqref{eq:gbound0} and \eqref{eq:taubound0} follow from 
\eqref{eq:psi0} and \eqref{eq:v0} in the same way. 
\qed

\end{document}